\documentclass{amsart}

\usepackage{mathrsfs}
\usepackage[dvips]{graphicx}
\usepackage{amsmath}
\usepackage[latin1]{inputenc}
\usepackage{amsfonts}
\usepackage{amssymb}
\usepackage{graphicx,epsfig}
\usepackage{amsthm}

\usepackage{amscd}
\usepackage{dsfont}
\usepackage[all,dvips]{xy}
\usepackage{fancyhdr}
\usepackage[T1]{fontenc}

\input cyracc.def

\newtheorem{theorem}{Theorem}[section]

\newtheorem{proposition}[theorem]{Proposition}

\theoremstyle{definition}
\newtheorem{definition}[theorem]{Definition}
\newtheorem{example}[theorem]{Example}

\theoremstyle{remark}
\newtheorem{remark}[theorem]{Remark}

\allowdisplaybreaks

\numberwithin{equation}{section}

\begin{document}

\title{Construction of Rota-Baxter algebras via Hopf module algebras}

\author{Run-Qiang Jian}
\address{School of Computer Science, Dongguan University
of Technology, 1, Daxue Road, Songshan Lake, 523808, Dongguan, P.
R. China}
%\curraddr{}
\email{jian.math@gmail.com}
\thanks{}

%    General info

\date{}

%\dedicatory{}


\begin{abstract}
We propose the notion of Hopf module algebras and show that the
projection onto the subspace of coinvariants is an idempotent
Rota-Baxter operator of weight -1. We also provide a construction
of Hopf module algebras by using Yetter-Drinfeld module algebras.
As an application, we prove that the positive part of a quantum
group admits idempotent Rota-Baxter algebra structures.
\end{abstract}

\maketitle
\section{Introduction}
Rota-Baxter operators originate from a work of G. Baxter on
probability theory in 1960 (\cite{Bax}). In that paper, Baxter
showed that some well-known identities in the theory of
fluctuations for random variables can be deduced from a simple
relation of operators on a commutative algebra. At the end of
1960s, based on this work and others, Rota introduced the notion
of Baxter algebra, which is called Rota-Baxter algebra nowadays in
honor of the contribution of Rota, in his fundamental papers
\cite{Rot}. Roughly speaking, a Rota-Baxter algebra is an
associative algebra together with a linear endomorphism which is
analogous to the integral operator. During the past four decades,
this algebraic object has been investigated extensively by many
mathematicians with various motivations. At present, Rota-Baxter
algebras have become a useful tool in many branches of
mathematics, such as combinatorics (\cite{EGP}), Loday type
algebras (\cite{E}, \cite{EG2} ), pre-Lie and pre-Poisson algebras
(\cite{LHB}, \cite{AB}, \cite{A}), multiple zeta values
(\cite{EG3}, \cite{GZ}), and so on. Besides their own interest in
mathematics, Rota-Baxter algebras also have many important
applications in mathematical physics. For instance, the theory of
non-commutative Rota-Baxter algebras with idempotent Rota-Baxter
operator is used to provide an algebraic setting in the Hopf
algebraic approach to the Connes-Kreimer theory of renormalization
in perturbative quantum field theory (cf. \cite{CK1}, \cite{CK2},
\cite{EGK1}, \cite{EGK2}).

Recently, in the work \cite{J}, we provide examples of Rota-Baxter
algebras from the so-called quantum quasi-shuffle algebras. It
leads to the notion of braided Rota-Baxter algebra. This also
brings our attention to the relation between Rota-Baxter algebras
and quantum algebras. On one hand, this relation will enable one
to use tools from Rota-Baxter algebras to study quantum algebras.
On the other hand, one would expect to provide more interesting
examples of Rota-Baxter algebras via quantized algebras. The
second point of view is the aim of the present paper. We first
attempt to construct idempotent Rota-Baxter operators for their
importance which has been mentioned at the end of the first
paragraph. In other words, we need some sort of good algebras and
projections. In order to provide such good things, we introduce
the notion of Hopf module algebra which mixes Hopf module
structure and algebra structure via some relevant compatibility
conditions, and consider the projection onto the subspace of
coinvariants. These new objects are very natural since they can be
constructed from the bosonization of algebras in the category of
Yetter-Drinfeld modules. A particular important example is the
bosonization of Nichols algebras which is closely related to
quantum groups. Hence, thanks to these relations, we can show that
the positive part of a quantum group admits idempotent Rota-Baxter
algebra structures.

This paper is organized as follows. In Section 2, we introduce the
notion of Hopf module algebra and provide a construction of
idempotent Rota-Baxter operators of weight -1 on these algebras.
Consequently, we provide a construction of Hopf module algebras
via Yetter-Drinfeld module algebras, as well as an application to
quantum groups, in Section 3.

\section{Hopf module algebras and the main construction}

Throughout this paper, we fix a ground field $\mathbb{K}$ of
characteristic 0 and assume that all vector spaces, algebras,
coalgebras and tensor products are defined over $\mathbb{K}$. An
algebra is always assumed to be associative, but not necessarily
unital.

We always denote by $(H,\Delta,\varepsilon, S)$ a Hopf algebra. We
adopt Sweedler's notation for coalgebras and comodules: for any
$h\in H$, denote
$$\Delta(h)=\sum h_{(1)}\otimes  h_{(2)},$$ for a left $H$-comodule $(M,\delta_L)$ and any $ m\in
M$, denote
$$\delta_L(m)=\sum m_{(-1)}\otimes  m_{(0)},$$ and for a right $H$-comodule $(N,\delta_R)$ and any $ n\in
N$, denote
$$\delta_R(n)=\sum n_{(0)}\otimes  n_{(1)}.$$

Let $M$ and $N$ be two left $H$-modules. As usual, we consider
$M\otimes N$ as a left $H$-module via the diagonal action
$h\cdot(m\otimes n)=\sum h_{(1)}\cdot m\otimes h_{(2)}\cdot n$ for
$h\in H$, $m\in M$ and $n\in N$. A similar convention is used for
right modules, left comodules and right comodules.

\begin{definition}[\cite{Sw}]A right (resp. left) \emph{$H$-Hopf module} is a vector space $M$ equipped simultaneously with a right (resp. left) $H$-module structure and a right (resp. left) $H$-comodule structure
$\delta_R$ (resp. $\delta_L$) such that $\delta_R( m \cdot
h)=\delta_R(m)\Delta(h)$ (resp. $\delta_L(h \cdot m
)=\Delta(h)\delta_L(m)$) whenever $h\in H$ and $m\in
M$.\end{definition}

Let $M$ be a right $H$-Hopf module. We denote by $M^R$ the
subspace of right coinvariants, i.e.,
$$M^R=\{m\in M|\ \delta_R(m)=m\otimes 1_H\}.$$ It is well known
that the map $P_R:M\rightarrow M$ given by $P_R(m)=\sum
m_{(0)}\cdot S(m_{(1)})$ for $m\in M$ is the projection from $M$
onto $M^R$ (cf. \cite{Sw}). Similarly, for a left $H$-Hopf module,
we can define the subspace $M^L$ of left coinvariants and have the
projection formula $P_L(m)=\sum S(m_{(-1)})\cdot m_{(0)}$.

Now we consider algebras in the category of $H$-Hopf modules.

\begin{definition}A right \emph{$H$-Hopf module
algebra} is a right $H$-Hopf module $M$ together with an
associative multiplication $\mu:M\otimes M\rightarrow M$ such that
for any $h\in H$ and $m,m'\in M$,
\[\left\{
\begin{split}
(mm')\cdot h&=m(m'\cdot h),\\[3pt]
\delta_R(mm')&=\sum m_{(0)} m'_{(0)}\otimes m_{(1)}m'_{(1)},
\end{split}\right.
\]where $mm'=\mu (m\otimes m')$.
\end{definition}

The compatibility conditions between $\mu$ and the Hopf module
structure mean that $\mu$ is both a module morphism and a comodule
morphism. We can define the left $H$-Hopf module algebra in a
similar way.

In the following, we will construct Rota-Baxter operators on Hopf
module algebras. We first recall the definition of Rota-Baxter
algebras.

\begin{definition}Let $\lambda$ be an element in $\mathbb{K}$. A pair $(R,P)$ is
called a \emph{Rota-Baxter algebra of weight $\lambda$} if $R$ is
an algebra and $P$ is a linear endomorphism of $R$ satisfying that
for any $x,y\in R$,
$$P(x)P(y)=P(xP(y))+P(P(x)y)+\lambda P(xy).$$The map $P$ is called a \emph{Rota-Baxter operator}.\end{definition}

A useful observation is that if $P$ is a Rota-Baxter operator of
weight $\lambda$ with $\lambda\neq 0$, then $\nu\lambda^{-1}P$ is
a Rota-Baxter operator of weight $\nu$ for arbitrary $\nu \in
\mathbb{K}$. So the weight is not "intrinsic" in some sense.

We can state our main construction now.

\begin{theorem}Let $(M,\mu)$ be a right $H$-Hopf module
algebra. Then $(M,P_R)$ is a Rota-Baxter algebra of weight -1.
\end{theorem}
\begin{proof}For any $m,m'\in M$, we have\begin{eqnarray*}
\lefteqn{P_R\big(mP_R(m')\big)+P_R\big(P_R(m)m'\big)-P_R(mm')}\\[3pt]
&=&\sum (mP_R(m'))_{(0)}\cdot S((mP_R(m'))_{(1)})+\sum (P_R(m)m')_{(0)}\cdot S((P_R(m)m')_{(1)})\\[3pt]
&&-\sum (mm')_{(0)}\cdot S((mm')_{(1)})\\[3pt]
&=&\sum (m_{(0)}P_R(m')_{(0)})\cdot S(m_{(1)}P_R(m')_{(1)})+\sum (P_R(m)_{(0)}m'_{(0)})\cdot S(P_R(m)_{(1)}m'_{(1)})\\[3pt]
&&-\sum (m_{(0)}m'_{(0)})\cdot S(m_{(1)}m'_{(1)})\\[3pt]
&=&\sum (m_{(0)}P_R(m'))\cdot S(m_{(1)})+\sum (P_R(m)m'_{(0)})\cdot S(m'_{(1)})\\[3pt]
&&-\sum (m_{(0)}m'_{(0)})\cdot S(m_{(1)}m'_{(1)})\\[3pt]
&=&P_R(m)\sum m'_{(0)}\cdot S(m'_{(1)})\\[3pt]
&&+\sum (m_{(0)}P_R(m'))\cdot S(m_{(1)})-\sum (m_{(0)}(m'_{(0)}\cdot S(m'_{(1)}))\cdot S(m_{(1)})\\[3pt]
&=&P_R(m)P_R(m').
\end{eqnarray*}\end{proof}

\begin{remark}(i) One can show that if an algebra can be decomposed into a direct sum, as vector spaces, of two subalgebras, then the projection onto one of them is an idempotent Rota-Baxter operator of weight -1 (cf. \cite{Guo}). So we can give another proof of the above theorem by showing that both $\mathrm{Im}P_R=M^R$ and $\mathrm{Ker}P_R$ are subalgebras.

(ii) Similarly, if $(M,\mu)$ is a left $H$-Hopf module algebra,
then one can show that $(M,P_L)$ is a Rota-Baxter algebra of
weight -1.
\end{remark}

\begin{example}[\cite{Ra}]Let $A$ be a bialgebra. Suppose there are two bialgebra maps $i:
H\rightarrow A$ and $\pi: A\rightarrow H$ such that $\pi\circ
i=\mathrm{id}_H$. Set $\Pi=\mathrm{id}_A\star(i\circ S\circ \pi)$,
where $\star$ is the convolution product on $\mathrm{End}(A)$.
Then $\Pi$ is an idempotent Rota-Baxter operator of weight -1. In
fact, $A$ is a right $H$-Hopf module algebra with the following
right $H$-Hopf module structure: for any $a\in A$ and $h\in H$,
\[\left\{
\begin{split}
a\cdot h&=ai(h),\\[3pt] \delta_R(a)&=\sum a_{(1)}\otimes \pi(a_{(2)}).
\end{split}\right.
\]It is not hard to see that the projection from $A$ to $A^R$ is
just $\Pi$. So by the above theorem, we get the conclusion.
\end{example}

The above construction contains many other interesting examples.
Let $H$ be a graded Hopf algebra. Then $H$ can be written as a
direct sum $H=\bigoplus_{k\geq 0}H_k$ such that $H_kH_l\subset
H_{k+l}$ and $\Delta(H_k)\subset \bigoplus_{r=0}^k H_r\otimes
H_{k-r}$. Obviously, the 0-component $H_0$ is a Hopf algebra with
the induced operations. Note that the inclusion $i$ from $H_0$
into $H$ and the projection $\pi$ from $H$ onto $H_0$ verify the
required conditions. Another important example is the multi-brace
cotensor Hopf algebra which is related to quantum groups (cf.
\cite{FR}).

\section{Construction of Hopf module algebras}

In order to give a construction of $H$-Hopf module algebras, we
need some necessary notions about Yetter-Drinfeld modules.

A \emph{left $H$-Yetter-Drinfeld module} is a vector space $V$
equipped simultaneously with a left $H$-module structure $\cdot$
and a left $H$-comodule structure $\rho$ such that whenever $h\in
H$ and $v\in V$,
$$\sum h_{(1)}v_{(-1)}\otimes h_{(2)}\cdot v_{(0)}=\sum
(h_{(1)}\cdot v)_{(-1)}h_{(2)}\otimes (h_{(1)}\cdot v)_{(0)}.
$$The category of left $H$-Yetter-Drinfeld modules, denoted by ${}^H_H\mathcal{YD}$,
consists of the following data: its objects are left
$H$-Yetter-Drinfeld modules and morphisms are linear maps which
are both module and comodule morphisms. An algebra $(V,\mu)$ in
${}^H_H\mathcal{YD}$ is an associative algebra such that the
underlying space $V$ is an object in ${}^H_H\mathcal{YD}$ and
$(V,\mu)$ is both a module-algebra and a comodule-algebra. More
precisely, if we denote $vv'=\mu(v\otimes v')$ for any $v,v'\in
V$, then$$h\cdot (vv')=\sum (h_{(1)}\cdot v)(h_{(2)}\cdot
v'),$$and
$$\rho(vv')=\sum v_{(-1)}v'_{(-1)}\otimes v_{(0)} v'_{(0)}.$$

Let $(V,\mu)$ be an algebra in ${}^H_H\mathcal{YD}$. The
\emph{smash product} $V\#H$ of $V$ and $H$ is defined to be
$V\otimes H$ as a vector space and
$$(v\# h)(v'\# h')=\sum v(h_{(1)}\cdot v')\# h_{(2)}h',\ \
v,v'\in V, h,h'\in H,$$where we use $\#$ instead of $\otimes$ to
emphasize this new algebra structure.

\begin{theorem}Let $(V,\mu)$ be an algebra in
${}^H_H\mathcal{YD}$.

(i) The algebra  $V\#H$ is a right $H$-Hopf module algebra with
the following right $H$-module structure: for any $h,x\in H$ and
$v\in V$,
$$(v\# h)\cdot x =v\#  hx,
$$and
$$
\delta_R(v\#  h)=\sum (v\#  h_{(1)})\otimes  h_{(2)}.
$$Hence the map $P_R$ given by $P_R(v\#  h)=v\#\varepsilon(h)1_H$
is an idempotent Rota-Baxter operator of weight -1.

(ii) The algebra $V\#H$ is a left $H$-Hopf module algebra with the
following left $H$-module structure:
$$x \cdot (v\#  h)=\sum x_{(1)}\cdot v\# x_{(2)}h,
$$and
$$
\delta_L(v\otimes h)=\sum v_{(-1)}h_{(1)}\otimes (v_{(0)}\#
h_{(2)}).
$$Hence the map $P_L$ given by $P_L(v\#  h)=\sum S(v_{(-1)}h_{(2)})\cdot v_{(0)}\#S(v_{(-2)}h_{(1)})h_{(3)}$
is an idempotent Rota-Baxter operator of weight -1.\end{theorem}
\begin{proof}The constructions of right and left Hopf module structures from left Yetter-Drinfeld modules given above are well-known. So the only thing to prove is the compatibility conditions.

(i) For any $v,v'\in V$ and $h,h',x\in H$, we
have\begin{eqnarray*}
((v\#h)(v'\#h'))\cdot x&=&\sum (v(h_{(1)}\cdot v')\# h_{(2)}h')\cdot x\\[3pt]
&=&\sum v(h_{(1)}\cdot v')\# h_{(2)}h' x\\[3pt]
&=&(v\#h)(v'\#h' x)\\[3pt]
&=&(v\#h)((v'\#h')\cdot x),
\end{eqnarray*}and\begin{eqnarray*}
\delta_R((v\#h)(v'\#h'))&=&\delta_R(\sum v(h_{(1)}\cdot v')\# h_{(2)}h')\\[3pt]
&=&\sum v(h_{(1)}\cdot v')\# h_{(2)}h'_{(1)}\otimes h_{(3)}h'_{(2)}\\[3pt]
&=&\sum (v\#h_{(1)})(v'\#h'_{(1)})\otimes h_{(2)}h'_{(2)}\\[3pt]
&=&\sum (v\#h)_{(0)}(v'\#h')_{(0)}\otimes
(v\#h)_{(1)}(v'\#h')_{(1)}.
\end{eqnarray*}

(ii) We have\begin{eqnarray*}
x\cdot((v\#h)(v'\#h')) &=&\sum x\cdot (v(h_{(1)}\cdot v')\# h_{(2)}h')\\[3pt]
&=&\sum x_{(1)}\cdot (v(h_{(1)}\cdot v')\# x_{(2)}h_{(2)}h')\\[3pt]
&=&\sum (x_{(1)}\cdot v)((x_{(2)}h_{(1)})\cdot v')\# x_{(3)}h_{(2)}h')\\[3pt]
&=&\sum ((x_{(1)}\cdot v)\# x_{(2)}h)( v'\# h')\\[3pt]
&=&(x\cdot(v\#h))(v'\#h'),
\end{eqnarray*}and\begin{eqnarray*}
\delta_L((v\#h)(v'\#h'))&=&\delta_L(\sum v(h_{(1)}\cdot v')\# h_{(2)}h')\\[3pt]
&=&\sum (v(h_{(1)}\cdot v'))_{(-1)}h_{(2)}h'_{(1)}\otimes (v(h_{(1)}\cdot v'))_{(0)}\#h_{(3)}h'_{(2)}\\[3pt]
&=&\sum v_{(-1)}(h_{(1)}\cdot v')_{(-1)}h_{(2)}h'_{(1)}\otimes (v_{(0)}(h_{(1)}\cdot v')_{(0)})\#h_{(3)}h'_{(2)}\\[3pt]
&=&\sum v_{(-1)}h_{(1)} v'_{(-1)}h'_{(1)}\otimes (v_{(0)}(h_{(2)}\cdot v'_{(0)}))\#h_{(3)}h'_{(2)}\\[3pt]
&=&\sum v_{(-1)}h_{(1)} v'_{(-1)}h'_{(1)}\otimes (v_{(0)}\#h_{(2)})(v'_{(0)}\# h'_{(2)})\\[3pt]
&=&\sum
(v\#h)_{(-1)}(v'\#h')_{(-1)}\otimes(v\#h)_{(0)}(v'\#h')_{(0)},
\end{eqnarray*}where the fourth equality follows from the Yetter-Drinfeld module condition.\end{proof}

We now give a concrete example of the above constructions which
helps us to obtain many Rota-Baxter algebras.

\begin{example}The algebra $H$ is an algebra in $ {}^H_H\mathcal{YD}$ with the
following left $H$-Yetter-Drinfeld module: for any $x,h\in H$,
\[
x\cdot h=\sum x_{(1)}h S(x_{(2)}),\ \ \rho(h)=\sum h_{(1)}\otimes
h_{(2)}.
\]Then the smash product on $H\# H$ is given by the following
formula:
$$(h^1\otimes h^2)(h^3\otimes h^4)=\sum
h^1h^2_{(1)}h^3S(h^2_{(2)})\otimes h^2_{(3)}h^4, h^i\in H.$$

Explicitly, the corresponding Rota-Baxter operators $P_R$ and
$P_L$ are given by $$P_R(h\#h')=h\#\varepsilon (h')1_H,$$ and
\begin{eqnarray*}
P_L(h\#h')&=&\sum S(h_{(2)}h'_{(2)})\cdot
h_{(3)}\#S(h_{(1)}h'_{(1)})h'_{(3)}\\[3pt]
&=&\sum S(h_{(3)}h'_{(3)})
h_{(4)}S^2(h_{(2)}h'_{(2)})\#S(h_{(1)}h'_{(1)})h'_{(4)}\\[3pt]
&=&\sum S(S(h_{(2)}h'_{(2)})h'_{(3)})
\#S(h_{(1)}h'_{(1)})h'_{(4)}.
\end{eqnarray*}
\end{example}

Finally, we apply our construction to Nichols algebras. Let $V$ be
a left $H$-Yetter-Drinfeld module and $R(V)=\bigoplus_{n\geq
0}R(n)$ be a graded Hopf algebra in ${}^H_H\mathcal{YD}$. We call
$R(V)$ a \emph{Nichols algebra} of $V$ if $\mathbb{K}\cong R(0)$
and $V\cong R(1)$ in ${}^H_H\mathcal{YD}$, and $R(V)$ is generated
as an algebra by $R(1)$ and the set of primitive elements is
exactly $R(1)$. These algebras play an important role in the
classification of pointed Hopf algebras (cf. \cite{AS2}). For more
detailed information about these algebras, we refer the reader to
\cite{AS}. Since $R(V)$ is an algebra in ${}^H_H\mathcal{YD}$, we
have Rota-Baxter algebra structures on $R(V)\#H$.

A particular important example is the quantum symmetric algebra.
Assume that $A=(a_{ij})_{1\leq i,j\leq N}$ is a symmetrizable
generalized Cartan matrix, and $(d_1,\ldots, d_N)$ are positive
relatively prime integers such that $(d_ia_{ij})$ is symmetric.
Let $q\in \mathbb{C}^\times$ and define $q_{ij}=q^{d_i a_{ij}}$.
Let $G=\mathbb{Z}^r\times \mathbb{Z}/l_1 \times\mathbb{Z}/l_2
\times\cdots \mathbb{Z}/l_p$ and $H=K[G]$ be the group algebra of
$G$. We fix generators $K_1,\ldots, K_N$ of $G$ ($N=r+p$). Denote
by $V$ the vector space over $\mathbb{C}$ with basis
$\{e_1,\ldots,e_N\}$. It is known that $V$ is an
$H$-Yetter-Drinfeld module with action and coaction given by
$K_i\cdot e_j=q_{ij}e_j$ and $\delta_L(e_i)=K_i\otimes e_i $
respectively. The quantum shuffle product $\ast$ on $T(V)$ is
defined by the following inductive formula: \begin{eqnarray*}
\lefteqn{(e_{i_1}\otimes \cdots \otimes
e_{i_n})\ast(e_{j_1}\otimes \cdots \otimes
e_{j_n})}\\[3pt]
&=&e_{i_1}\otimes((e_{i_2}\otimes \cdots \otimes
e_{i_n})\ast(e_{j_1}\otimes \cdots \otimes
e_{j_n}))\\[3pt]
&&+q_{i_1j_1}\cdots q_{i_nj_1}e_{j_1}\otimes((e_{i_1}\otimes
\cdots \otimes e_{i_n})\ast(e_{j_2}\otimes \cdots \otimes
e_{j_n})).
\end{eqnarray*}Then the subalgebra $S(V)$ generated by $V$ is a Nichols algebra. By Theorem 15 in \cite{Ro}, $S_H(M)=S(V)\#H$ is
isomorphic, as a Hopf algebra, to the sub Hopf algebra $U_q^+$ of
the quantized universal enveloping algebra associated with $A$
when $G=\mathbb{Z}^N$ and $q$ is not a root of unity; $S_H(M)$ is
isomorphic, as a Hopf algebra, to the quotient of the restricted
quantized enveloping algebra $u_q^+$ by the two-sided Hopf ideal
generated by the elements $(K_i^l-1)$, $i=1,\ldots, N$ when
$G=(\mathbb{Z}/l)^N$ and $q$ is a primitive $l$-th root of unity.
Thus we obtain that

\begin{proposition}Both $U_q^+$ and $u_q^+$ possess idempotent Rota-Baxter algebra structures. \end{proposition}

\section*{Acknowledgements}
This work was conceived during a visit to the Department of
Mathematics of Universi\'{e} Denis Diderot-Paris 7 in 2013. I am
grateful to Marc Rosso for his invitation and hospitality. I thank
Xin Fang for communications. Thanks also go to Leandro Vendramin
for his interest on this work and pointing out some typos in an
earlier version. Finally, I thank the referees for their careful
reading and comments. This work was partially supported by the
National Natural Science Foundation of China (Grant No. 11201067)
and a grant from DGUT (Grant No. ZF121006).

\bibliographystyle{amsplain}

\end{document}